\documentclass[11pt]{article}
\usepackage[english]{babel}
\usepackage[babel]{microtype}
\usepackage[margin=2.3cm]{geometry}
\usepackage{amsmath,amssymb,amsfonts,amsthm,mathtools,bm,mathrsfs,bbm}
\usepackage[shortlabels]{enumitem}
\usepackage{needspace}
\usepackage[hidelinks]{hyperref}
\newtheorem{theorem}{Theorem}[section]
\newtheorem{lemma}[theorem]{Lemma}
\newtheorem{prop}[theorem]{Proposition}
\newtheorem{remark}[theorem]{Remark}
\newtheorem{cor}[theorem]{Corollary}

\newcommand{\E}{\mathbb{E}}
\newcommand{\Prb}{\mathbb{P}}
\newcommand{\Inf}{\mathrm{Inf}}
\newcommand{\Var}{\mathrm{Var}}
\newcommand{\Cov}{\mathrm{Cov}}
\newcommand{\W}{\mathcal W}
\newcommand{\cube}{\{0,1\}}
\newcommand{\tup}[1]{\left\langle #1\right\rangle}
\DeclareMathOperator{\supp}{supp}
\newcommand{\eqtag}[1]{\refstepcounter{equation}\label{#1}\textup{(\theequation)}}

\title{A proof of Chv\'atal's conjecture via a sharp correlation inequality}
\author{Fan Chang\thanks{School of Statistics and Data Science, Nankai University, Tianjin, China. Email: \texttt{1120230060@mail.nankai.edu.cn}. Supported by the National Natural Science Foundation of China (NSFC) under grant 124B2019.}, \quad 
Hong Liu\thanks{Extremal Combinatorics and Probability Group (ECOPRO), Institute for Basic Science (IBS), Daejeon, South Korea. Email: \texttt{hongliu@ibs.re.kr}. Supported by Institute for Basic Science IBS-R029-C4.},\quad
Miao Liu\thanks{Research Center for Mathematics and Interdisciplinary Sciences, Shandong University, Qingdao, China. Email: \texttt{liumiao10300403@163.com}. Supported by the National Natural Science Foundation of China under Grant Nos.~12571352 and 12231014.}
}
\date{}
\begin{document}
\maketitle

\begin{abstract}
We prove Chv\'atal's conjecture, posed in 1972: every hereditary family of subsets of a finite set has a largest intersecting subfamily that is a star. More generally, we prove a sharp correlation inequality for increasing
Boolean functions $f,g:\{0,1\}^n\to\{0,1\}$. Writing $g^*(x)=1-g(1-x)$, we show that
$$
\sum_{\varnothing\ne S\subseteq[n]}
\hat{g}(S)^2\max_{i\in S}\mathrm{Inf}_i[f]\le\frac{2\mathrm{Cov}(f,g)\mathrm{Cov}(f,g^*)}{\mathrm{Cov}(f,g)+\mathrm{Cov}(f,g^*)}.
$$
When $g$ is antipodal, that is, $g=g^*$,
this yields $\mathrm{Cov}(f,g)\ge\frac{1}{4}\min_{i\in[n]}\mathrm{Inf}_i[f]$, the correlation formulation of Chv\'atal's conjecture due to
Friedgut, Kahn, Kalai and Keller.
\end{abstract}

\section{Introduction}\label{sec:intro}

Chv\'atal conjectured in 1972~\cite[Problem~25]{CKK1972} (see also~\cite{C1974}) that some largest intersecting subfamily of every hereditary family is a star. Here a family $\mathcal{D}\subseteq 2^{[n]}$ is \emph{hereditary}, or \emph{decreasing}, if it is closed under taking subsets. A family $\mathcal{F}$ is \emph{intersecting} if $A\cap B\ne\varnothing$ for all $A,B\in\mathcal{F}$. For $i\in[n]$, the \emph{star} of $\mathcal{D}$ centred at $i$ is $\mathcal{D}(i)=\{G\in\mathcal{D}:i\in G\}$.  

We prove this conjecture. 
\begin{theorem}\label{conj:chvatal}
Let $\mathcal{D}\subseteq 2^{[n]}$ be decreasing. If $\mathcal{F}\subseteq \mathcal{D}$ is intersecting, then there exists $k\in[n]$ such that
$$
|\mathcal{F}|\le\bigl|\mathcal{D}(k)\bigr|.
$$
\end{theorem}
We briefly recall some earlier results. Berge~\cite{Berge1976} proved that either $\mathcal{D}$ or $\mathcal{D}\setminus\{\varnothing\}$ can be partitioned into pairs
of disjoint sets, giving the bound $|\mathcal F|\le|\mathcal{D}|/2$. Chv\'atal~\cite{C1974} established the conjecture for
left-compressed families, and Snevily~\cite{Snevily1992} extended
this to families that are $ij$-compressed for every $j\ne i$,
for some fixed $i$.
For a uniform variant, Kupavskii~\cite{Kupavskii2023}
proved that a largest $t$-intersecting subfamily of the $k$th layer of $\mathcal{D}$ can be chosen to be a $t$-star if every inclusion-maximal member of $\mathcal{D}$ has size at least $Ctk\log^2(2k)$, for an absolute constant $C$. More recently, Looney et al.~\cite{LMKWPW2026} gave a computer-assisted proof of the conjecture for $n\le8$.

Our approach uses the correlation formulation of Friedgut, Kahn, Kalai and Keller~\cite{FKKK2018correlation}. Equip $\cube^n$ with the uniform probability measure. For Boolean functions $f,g:\cube^n\to \cube$, we write $\Cov(f,g)=\E[fg]-\E[f]\E[g]$ and $\Inf_i[f]=\Prb[f(X)\ne f(X\oplus \mathbf{1}_{\{i\}})]$, where $X$ is uniform on $\{0,1\}^n$, $\mathbf{1}_{\{i\}}$ is the indicator vector of $\{i\}$, and $\oplus$ denotes addition modulo $2$. A function is \emph{increasing} if it is nondecreasing in each coordinate. The \emph{dual} of $g$ is $g^*(x)=1-g(1-x)$; the function $g$ is \emph{antipodal} if $g=g^*$. We use the Fourier convention $\hat g(S)=\E[g(X)\chi_S(X)]$, where $\chi_S(x)=(-1)^{\sum_{i\in S}x_i}$. The main analytic result is the following.

\begin{theorem}\label{thm:harmonic}
Let $f,g:\cube^n\to\cube$ be increasing. Then\footnote{The right-hand side is defined to be zero when the denominator vanishes.}
$$
\sum_{\varnothing\ne S\subseteq[n]}\hat g(S)^2\max_{i\in S}\Inf_i[f]
\le
\frac{2\Cov(f,g)\Cov(f,g^*)}{\Cov(f,g)+\Cov(f,g^*)}.
\eqno\eqtag{eq:harmonic}
$$
\end{theorem}
Both covariances are nonnegative by the Harris--Kleitman inequality~\cite{Harris1960,Kleitman1966}, so the right-hand side is their harmonic mean. If $g$ is antipodal, then $g=g^*$ and $\sum_{\varnothing\ne S\subseteq[n]}\hat g(S)^2=\Var(g)=\frac{1}{4}$. We therefore obtain the following consequence.

\begin{cor}[conjectured by Friedgut--Kahn--Kalai--Keller~\cite{FKKK2018correlation}]
\label{conj:FKKK-antipodal}
Let $f,g:\{0,1\}^n\to\{0,1\}$ be increasing, and suppose that $g$ is antipodal. Then
\begin{equation}\label{eq:antipodal-consequence}
\Cov(f,g)
\ge\frac14\min_{i\in[n]}\Inf_i[f].
\end{equation}    
\end{cor}
Friedgut, Kahn, Kalai and Keller showed that this inequality is equivalent to Chvátal's conjecture~\cite[Proposition 2.1]{FKKK2018correlation}. For completeness, we give the counting argument in Section~\ref{sec:known-equivalence}.

Duality provides a natural way to capture the symmetry of antipodality in a more general setting as if $g$ is antipodal, $\Cov(f,g)$ and $\Cov(f,g^*)$ coincide. Theorem~\ref{thm:harmonic} follows from the following stronger inequality. The parameter $t$ balances the contributions of $g$ and its dual. 

\begin{theorem}\label{thm:cross}
Let $f,g:\cube^n\to\cube$ be increasing. For every $t\in\mathbb R$,
$$
2\sum_{\substack{S,T\subseteq[n]\\|S\cap T|\text{ odd}}}
\hat f(S)^2\hat g(T)^2
\le t^2\Cov(f,g)+(1-t)^2\Cov(f,g^*).
\eqno\eqtag{eq:cross-parameter}
$$
\end{theorem}
To deduce Theorem~\ref{thm:harmonic}, we first use Lemma~\ref{lem:flip} below to obtain
$$
\sum_{\varnothing\ne T\subseteq[n]}
\hat g(T)^2\max_{i\in T}\Inf_i[f]
\le
4\sum_{\substack{S,T\subseteq[n]\\|S\cap T|\text{ odd}}}
\hat f(S)^2\hat g(T)^2.
\eqno\eqtag{eq:influence-to-parity}
$$
If $\Cov(f,g)+\Cov(f,g^*)>0$, the right-hand side of \eqref{eq:cross-parameter} is minimized at
$$
t=\frac{\Cov(f,g^*)}{\Cov(f,g)+\Cov(f,g^*)},
$$
with minimum value $\Cov(f,g)\Cov(f,g^*)/(\Cov(f,g)+\Cov(f,g^*))$.
Multiplying the resulting bound by two and applying \eqref{eq:influence-to-parity} gives \eqref{eq:harmonic}. The inequality is sharp: taking $f(x)=\prod_{i=1}^n x_i$ gives equality in~\eqref{eq:harmonic} for every increasing Boolean $g$. Section~\ref{sec:consequences} gives the calculation and further consequences.

\paragraph{Relation to earlier correlation inequalities.} Talagrand~\cite{Talagrand1996correlated} initiated the quantitative study of positive correlation between increasing families. Keller, Mossel and Sen~\cite{KMS2014correlation} obtained a complementary coordinatewise influence bound, and Kalai, Keller and
Mossel~\cite{KKM2016correlation} established refinements under regularity and symmetry assumptions. Eldan~\cite{Eldan2022} developed second-order Fourier refinements of the Talagrand and Keller--Mossel--Sen bounds. In particular, when one function is antipodal, his result improves the logarithmic loss in Talagrand's inequality to its square root. In a spectral direction, Chang and Chen~\cite{CC2027}
established log-free influence and diagonal spectral inequalities under submodularity or supermodularity assumptions. Chang~\cite{Chang2026spectral} obtained a general diagonal spectral
bound, and Chang, Liu and Liu~\cite{CLL2026spectral} subsequently
proved the sharp diagonal inequality using a correlated induction argument.

\paragraph{Proof overview.}
Write $\mathcal F=\{f=1\}$, $\mathcal G=\{g=1\}$, and $\mathcal G^*=\{g^*=1\}$, identifying cube points with subsets of $[n]$. The main difficulty in Theorem~\ref{thm:cross} is the mixture of two Fourier spectra under the odd-intersection condition. The Fourier identity in Lemma~\ref{lem:flip} and the fact that $f$ is Boolean-valued gives
$$
2\sum_{\substack{S,T\subseteq[n]\\|S\cap T|\text{ odd}}}\hat{f}(S)^2\hat{g}(T)^2=\frac{1}{2^n}\sum_{\substack{x\in\mathcal{F}\\y\notin\mathcal{F}}}\hat{g}(x\oplus y)^2.
$$
This removes the odd-intersection restriction and retains $f$ only through the summation domain. Moreover, every index $x\oplus y$ in this sum is nonempty. Replacing $g$ by $g-t$ therefore leaves the sum unchanged, allowing us to introduce the free parameter from the desired covariance bound.

Parseval's identity gives
\begin{equation}\label{eq:Fourier-decomp}
\sum_{\substack{x\in\mathcal{F}\\y\notin\mathcal{F}}}\hat{g}(x\oplus y)^2
=|\mathcal{F}|\|g-t\|_2^2-\sum_{x,y\in\mathcal{F}}\widehat{g-t}(x\oplus y)^2.
\end{equation}
Hence it suffices to find a lower bound for $\sum_{x,y\in\mathcal{F}}\widehat{g-t}(x\oplus y)^2$. We seek auxiliary functions supported on $\mathcal F$ whose Fourier supports lie entirely outside or inside $\mathcal G$, thereby separating the two values $-t$ and $1-t$ of $g-t$. The monomials $M_S(x)=\prod_{i\in S}x_i$ are suited to these requirements: their supports consist of up-sets of $S$, whereas their Fourier supports consist of subsets of $S$. Monotonicity thus gives the first family $\{M_S:S\in\mathcal F\setminus\mathcal G\}$. Multiplication by $\chi_{[n]}$ preserves support and complements Fourier indices, giving the second family $\{\chi_{[n]}M_S:S\in\mathcal F\setminus\mathcal G^*\}$ with Fourier support inside $\mathcal G$.

Each family is linearly independent, and their disjoint Fourier supports make the two families mutually orthogonal. Gram--Schmidt orthogonalization preserves all these support conditions. Applying Bessel's inequality to $H_y(x)=2^n f(x)\widehat{g-t}(x\oplus y)$ with the resulting orthonormal system, and summing over $y\in\mathcal F$, yields
$$
\sum_{x,y\in\mathcal F}\widehat{g-t}(x\oplus y)^2
\ge
t^2|\mathcal F\setminus\mathcal G|
+
(1-t)^2|\mathcal F\setminus\mathcal G^*|.
$$
Substituting this bound into~\eqref{eq:Fourier-decomp} gives precisely $t^2\Cov(f,g)+(1-t)^2\Cov(f,g^*)$, proving~\eqref{eq:cross-parameter}.

\medskip
\noindent\emph{Organization.}
Section~\ref{sec:preliminaries} reviews Fourier analysis on the discrete
cube and Bessel's inequality. Section~\ref{sec:proofs} proves Theorem~\ref{thm:cross}, Theorem~\ref{thm:harmonic} and Corollary~\ref{conj:FKKK-antipodal}. Section~\ref{sec:known-equivalence} recalls the equivalence between
Chv\'atal's conjecture and Corollary~\ref{conj:FKKK-antipodal}. Section~\ref{sec:consequences} discusses sharpness and further consequences.

\section{Preliminaries}\label{sec:preliminaries}

We identify a point $x\in\cube^n$ with the set $\{i\in [n]:x_i=1\}$. In particular, $g(S)$ denotes the value of $g$ at the indicator vector of $S$, whereas $\hat g(S)$ is its Fourier
coefficient. For $T\subseteq[n]$, let $\mathbf{1}_T$ be its indicator vector. The expression $x\oplus \mathbf{1}_T$ denotes flipping all coordinates of $x$ in $T$,
so $x\oplus y$ corresponds to symmetric difference. The complement of $S\subseteq[n]$ is $[n]\setminus S$, and $1-x$ denotes the coordinatewise complement of $x$.

For real-valued functions $f,g:\cube^n\to\mathbb{R}$, let $\tup{f,g}=\E[fg]$ and $\|f\|_2^2=\E[f^2]$. For $S\subseteq[n]$, define the Fourier--Walsh character by
$\chi_S(x):=(-1)^{\sum_{i\in S}x_i}.$
The family $\{\chi_S\}_{S\subseteq[n]}$ is an orthonormal basis of $L^2(\cube^n)$. The Fourier--Walsh expansion of $f:\cube^n\to\mathbb R$ is
\[
f(x)=\sum_{S\subseteq[n]}\hat f(S)\chi_S(x),
\qquad
\hat f(S)=\tup{f,\chi_S}.
\]
For more background on the Fourier–Walsh expansion the reader is referred to~\cite{ryanbook2014}. 

Parseval’s identity gives
$$
\tup{f,g}=\sum_{S\subseteq[n]}\hat f(S)\hat g(S).
$$
In particular, $\|f\|_2^2=\sum_S\hat f(S)^2$ and $\Var(f)=\sum_{S\ne\varnothing}\hat f(S)^2$. 

The support of $f$ is $\supp f=\{x\in \{0,1\}^n:f(x)\ne0\}$, and its Fourier support is $\supp\hat{f}=\{S\subseteq [n]:\hat{f}(S)\ne0\}$. A Boolean function is identified with its support family. If $g$ is increasing, then its dual $g^*$ is increasing and $\E[g^*]=1-\E[g]$. The corresponding dual family is $\mathcal{G}^*=\{S\subseteq[n]:[n]\setminus S\notin\mathcal G\}$.

We will also use the elementary influence identity
$$
\Inf_i[f]=2\E[(2X_i-1)f(X)]
\eqno\eqtag{eq:influence-sign}
$$
for an increasing Boolean function $f$. Here $X$ is uniform on $\{0,1\}^n$. Write $x_{-i}$ for the coordinates of $x$ other than $i$, and let $f(x_{-i},b)$ denote the value obtained by setting the $i$th coordinate to $b$. Indeed, after fixing the other coordinates, $f(x_{-i},1)-f(x_{-i},0)$ is either $0$ or $1$. Its expectation is both $\Inf_i[f]$ and the right-hand side of~\eqref{eq:influence-sign}. 

We next recall the classical finite-dimensional form of Bessel's inequality for real-valued functions. 

\begin{theorem}[Bessel's inequality]\label{thm:bessel}
Let $\phi_1,\ldots,\phi_m$ be orthonormal real-valued functions on $\cube^n$. For every real-valued function $h$ on $\cube^n$,
$$
\sum_{j=1}^m\tup{h,\phi_j}^2\le\|h\|_2^2.
$$
The assertion also holds for $m=0$, when the sum is zero.
\end{theorem} 

\section{Proofs of the main results}\label{sec:proofs}

We first construct the auxiliary functions used in the proof of Theorem~\ref{thm:cross}. We then prove that theorem and deduce Theorem~\ref{thm:harmonic} directly from it. 

\subsection{Auxiliary functions}\label{subsec:auxiliary}

Let $f,g$ be increasing Boolean functions, and write $\mathcal F=\supp f$, $\mathcal G=\supp g$ and $\mathcal G^*=\supp g^*$. Fix $t\in \mathbb{R}$. The proof of Theorem~\ref{thm:cross} will reduce to a lower bound for
$
\sum_{x,y\in\mathcal F}\widehat{g-t}(x\oplus y)^2.
$
For each $y\in\mathcal F$, define $H_y(x)=2^nf(x)\widehat{g-t}(x\oplus y)$. Then 
\begin{equation}\label{eq:kernel-vector}
\sum_{x,y\in\mathcal F}\widehat{g-t}(x\oplus y)^2
=\frac1{2^n}\sum_{y\in\mathcal F}\|H_y\|_2^2.
\end{equation}
Thus we need a lower bound for a sum of squared norms. Bessel's inequality gives such a bound by testing against orthonormal functions. Our task is to find enough test functions for which the inner products with $H_y$ can be computed exactly. 

For any real-valued function $h$ supported on $\mathcal{F}$, expanding the Fourier coefficient gives
\begin{equation*}
\begin{aligned}
\tup{H_y,h}=\sum_{x\in\mathcal F}h(x)\widehat{g-t}(x\oplus y)&=\sum_{x\in \mathcal F}h(x)\E_{S}\left[(g(S)-t)\chi_{x\oplus y}(S)\right] \\
&=\frac1{2^n}\sum_{S\subseteq[n]}(g(S)-t)\chi_S(y)\sum_{x\in\mathcal F}h(x)\chi_S(x)\\
&=\sum_{S\subseteq[n]}(g(S)-t)\hat h(S)\chi_S(y).
\end{aligned}
\end{equation*}
Here we used $\chi_{x\oplus y}(S)=\chi_S(x)\chi_S(y)$ and the fact that $h$ vanishes outside $\mathcal{F}$. Fourier inversion therefore yields
\begin{equation}\label{eq:aux-testing}
\tup{H_y,h}=
\begin{cases}
-t\,h(y),&\text{if }\supp\hat h\subseteq 2^{[n]}\setminus\mathcal{G},\\[2pt]
(1-t)h(y),&\text{if }\supp\hat h\subseteq\mathcal{G}.
\end{cases}
\end{equation}
This explains both support requirements: the support keeps $h$ inside $\mathcal{F}$, while the Fourier support condition makes $g-t$ constant on $\supp \hat{h}$, with value $-t$ or $1-t$. 

The required functions come from monomials. For $S\subseteq[n]$, define  $M_S(x)=\prod_{i\in S}x_i$, with $M_\varnothing=1$. Equivalently, $M_S(x)=\mathbbm{1}_{\{S\subseteq x\}}$. Since $x_i=\frac{1-\chi_{\{i\}}(x)}{2}$, its Fourier expansion is
$$
M_S(x)=2^{-|S|}\sum_{T\subseteq S}(-1)^{|T|}\chi_T(x).
\eqno\eqtag{eq:monomial-fourier}
$$
In particular,
$$
\supp M_S=\{x\in \{0,1\}^n:S\subseteq x\},
\qquad
\supp \widehat{M_S}=\{T\subseteq [n]:T\subseteq S\}. 
$$
If $S\in \mathcal{F}\setminus\mathcal{G}$, upward closure of $\mathcal{F}$ gives $\supp M_S\subseteq \mathcal{F}$, whereas the downward closure of $2^{[n]}\setminus \mathcal{G}$ gives  $\supp \widehat{M_S}\subseteq 2^{[n]}\setminus \mathcal G$. 

Multiplication by $\chi_{[n]}$ leaves the support unchanged and replaces each Fourier index by its complement. More explicitly,
$$
\chi_{[n]}M_S
=2^{-|S|}\sum_{T\subseteq S}(-1)^{|T|}\chi_{[n]\setminus T},
\qquad
\supp \widehat{\chi_{[n]}M_S}=\{R\subseteq [n]:[n]\setminus S\subseteq R\}.
$$
When $S\in \mathcal{F}\setminus \mathcal{G}^*$, we have $[n]\setminus S\in \mathcal{G}$. The upward closure of $\mathcal{G}$ gives $\supp \widehat{\chi_{[n]}M_S}\subseteq\mathcal{G}$, while $\supp (\chi_{[n]}M_S)\subseteq \mathcal{F}$.

\begin{lemma}\label{lem:auxiliary-families}
Let $\mathcal F,\mathcal G\subseteq2^{[n]}$ be increasing families, and let $\mathcal G^*$ be the dual of $\mathcal G$. Define the following two families of real-valued functions:
$$
\mathcal F'=\{M_S:S\in\mathcal F\setminus\mathcal G\},
\qquad
\mathcal G'=\{\chi_{[n]}M_S:S\in\mathcal F\setminus\mathcal G^*\}.
$$
They have the following properties.
\begin{enumerate}[label=\textup{(\roman*)},leftmargin=2.2em]
\item Every $h\in\mathcal F'$ satisfies $\supp h\subseteq\mathcal F$ and $\supp\hat h\subseteq2^{[n]}\setminus\mathcal G$.
\item Every $h\in\mathcal G'$ satisfies $\supp h\subseteq\mathcal F$ and $\supp\hat h\subseteq\mathcal G$.
\item Each of $\mathcal F'$ and $\mathcal G'$ is linearly independent.
\item Every member of $\mathcal F'$ is orthogonal to every member of $\mathcal G'$.
\end{enumerate}
\end{lemma}

\begin{proof}
If $S\in\mathcal F$, every point in the support of $M_S$ contains $S$ and hence lies in $\mathcal F$. Multiplication by $\chi_{[n]}$, which takes values in $\{-1,1\}$, does not change its support.

If $S\notin\mathcal G$, every subset of $S$ also lies outside $\mathcal G$, because $\mathcal G$ is increasing. Equation~\eqref{eq:monomial-fourier} therefore proves the assertion about Fourier support in \textup{(i)}. For \textup{(ii)}, the condition $S\notin\mathcal G^*$ means $[n]\setminus S\in\mathcal G$. Since $\chi_{[n]}\chi_T=\chi_{[n]\setminus T}$, the Fourier indices of $\chi_{[n]}M_S$ are exactly $[n]\setminus T$ for $T\subseteq S$. Each such index contains $[n]\setminus S$ and therefore belongs to $\mathcal G$.

To prove independence, consider any distinct sets $S_1,\ldots,S_r$, ordered so that $|S_1|\le\cdots\le|S_r|$. Suppose $\sum_{j=1}^r c_jM_{S_j}=0$. At the point $S_1$, the only monomial among this list that is nonzero is $M_{S_1}$, so $c_1=0$. Inductively, suppose $c_1=\cdots=c_{k-1}=0$ and evaluate at $S_k$. If $j>k$, then $S_j\nsubseteq S_k$: its cardinality is at least $|S_k|$, and the sets are distinct. Thus only $c_kM_{S_k}(S_k)=c_k$ remains, giving $c_k=0$. This proves independence of every such monomial family, in particular $\mathcal F'$. Multiplying a relation among members of $\mathcal G'$ by $\chi_{[n]}$ reduces it to a relation among distinct monomials, proving independence of $\mathcal G'$.

Finally, the Fourier supports in \textup{(i)} and \textup{(ii)} are disjoint. Parseval's identity therefore gives $\tup{h,k}=\sum_T\hat h(T)\hat k(T)=0$ whenever $h\in\mathcal F'$ and $k\in\mathcal G'$. This proves \textup{(iv)}.
\end{proof}

We next pass to orthonormal families without changing either support condition. 
\begin{cor}\label{cor:orthonormal-families}
Under the hypotheses of Lemma~\ref{lem:auxiliary-families}, let $r=|\mathcal F\setminus\mathcal G|$ and $s=|\mathcal F\setminus\mathcal G^*|$. There exist real-valued functions $u_1,\ldots,u_r,v_1,\ldots,v_s$ such that:
\begin{enumerate}
[label=\textup{(\roman*)},leftmargin=2.2em]
\item every function is supported on $\mathcal F$;
\item $\supp\hat u_j\subseteq2^{[n]}\setminus\mathcal G$ for $1\le j\le r$, and $\supp\hat v_k\subseteq\mathcal G$ for $1\le k\le s$;
\item $u_1,\ldots,u_r,v_1,\ldots,v_s$
is orthonormal and each has unit norm. 
\end{enumerate}
\end{cor}
\begin{proof}
    If $r>0$, enumerate $\mathcal F'$ as $a_1,\ldots,a_r$. Apply the Gram--Schmidt procedure by setting, for $1\le k\le r$, 
    $$w_k=a_k-\sum_{j=1}^{k-1}\tup{a_k,u_j}u_j, \qquad u_k=\frac{w_k}{\|w_k\|_2}, $$
    where the sum is empty when $k=1$. 
    We have $w_k\not= 0$; otherwise, $a_k$ would lie in the span of $a_1,\dots,a_{k-1}$, contrary to Lemma~\ref{lem:auxiliary-families}\textup{(iii)}. Thus all normalizations are valid. The resulting functions $u_1,\dots,u_r$ are pairwise orthogonal and have unit norm. Each $u_k$ is a linear combination of members of $\mathcal{F}'$, so Lemma~\ref{lem:auxiliary-families}\textup{(i)} ensures that $\supp u_k\subseteq\mathcal F$ and $\supp \hat{u}_k\subseteq 2^{[n]}\setminus \mathcal{G}$. If $r=0$, take the first list to be empty.  

    Applying the same procedure to $\mathcal{G}'$ gives an orthonormal list  $v_1,\dots,v_s$ satisfying the support conditions in Lemma~\ref{lem:auxiliary-families}\textup{(ii)}; the list is empty when $s=0$. By Lemma~\ref{lem:auxiliary-families}\textup{(iv)}, the spans of $\mathcal{F}'$ and $\mathcal{G}'$ are orthogonal. Since each $u_j$ belongs to the first span and each $v_k$ to the second, we have $\tup{u_j,v_k}=0$ for all $1\le j\le r$ and $1\le k\le s$. 
\end{proof}

\subsection{Proof of Theorem~\ref{thm:cross}}

\begin{proof}
Let $\mathcal F=\supp f$, $\mathcal G=\supp g$ and $\mathcal G^*=\supp g^*$. We prove~\eqref{eq:cross-parameter} for an arbitrary $t\in \mathbb{R}$. 

For every $T\subseteq [n]$, the identity $\chi_S(x\oplus \mathbf{1}_T)=(-1)^{|S\cap T|}\chi_S(x)$ and Parseval's identity give
\begin{equation}\label{eq:flip-fourier}
 \E_x\bigl[(f(x)-f(x\oplus \mathbf{1}_T))^2\bigr]=\sum_{S\subseteq[n]}\bigl(1-(-1)^{|S\cap T|}\bigr)^2\hat f(S)^2=4\sum_{\substack{S\subseteq[n]\\|S\cap T|\text{ odd}}}\hat f(S)^2. 
\end{equation}
Multiplying by $\hat{g}(T)^2/2$, summing over $T$, and changing variables to $y=x\oplus \mathbf{1}_T$, we obtain  
$$
\begin{aligned}
2\sum_{\substack{S,T\subseteq[n]\\|S\cap T|\text{ odd}}}\hat f(S)^2\hat g(T)^2
&=\frac12\sum_{T\subseteq[n]}\hat g(T)^2
\E_x\bigl[(f(x)-f(x\oplus \mathbf{1}_T))^2\bigr]\\
&=\frac1{2^{n+1}}\sum_{x,y\in\cube^n}
(f(x)-f(y))^2\hat g(x\oplus y)^2\\
&=\frac1{2^n}\sum_{\substack{x\in\mathcal F\\y\notin\mathcal F}}\hat g(x\oplus y)^2,
\end{aligned}
\eqno\eqtag{eq:cross-support}
$$
where the last equality uses the fact that $f$ is Boolean-valued and the symmetry between $(x,y)$ and $(y,x)$.

Subtracting the constant function $t$ changes only the constant Fourier coefficient:
$$
\widehat{g-t}(R)=\hat g(R)-t\mathbbm{1}_{\{R=\varnothing\}}.
$$
Every pair in~\eqref{eq:cross-support} has $x\ne y$, hence $x\oplus y\ne\varnothing$. Thus
$$
\sum_{\substack{x\in\mathcal F\\y\notin\mathcal F}}\hat g(x\oplus y)^2
=\sum_{\substack{x\in\mathcal F\\y\notin\mathcal F}}\widehat{g-t}(x\oplus y)^2.
\eqno\eqtag{eq:center-cross}
$$
This equality holds for every real $t$. For each fixed $x$, the map $y\mapsto x\oplus y$ enumerates all Fourier indices. Parseval's identity therefore gives $\sum_y\widehat{g-t}(x\oplus y)^2=\|g-t\|_2^2$. Splitting that sum according to membership in $\mathcal F$ yields
$$
\sum_{\substack{x\in\mathcal F\\y\notin\mathcal F}}\widehat{g-t}(x\oplus y)^2
=|\mathcal F|\,\|g-t\|_2^2
 -\sum_{x,y\in\mathcal F}\widehat{g-t}(x\oplus y)^2.
\eqno\eqtag{eq:split-energy}
$$
It therefore suffices to give a lower bound on the last sum. 

For each $y\in\mathcal F$, use the function $H_y(x)=2^n\mathbbm{1}_{\mathcal F}(x)\widehat{g-t}(x\oplus y)$ defined in Section~\ref{subsec:auxiliary}. Take $u_1,\ldots,u_r,v_1,\ldots,v_s$ from Corollary~\ref{cor:orthonormal-families}, where $r=|\mathcal F\setminus\mathcal G|$ and $s=|\mathcal F\setminus\mathcal G^*|$. Equation~\eqref{eq:aux-testing} implies that
$$
\tup{H_y,u_j}=-{t}u_j(y),
\qquad
\tup{H_y,v_k}=(1-t) v_k(y),
$$
Applying Theorem~\ref{thm:bessel} to $H_y$ with $u_1,\ldots,u_r,v_1,\ldots,v_s$, and~\eqref{eq:kernel-vector} gives
$$
t^2\sum_{j=1}^r u_j(y)^2
+{(1-t)^2}\sum_{k=1}^s v_k(y)^2
\le\|H_y\|_2^2={2^n}\sum_{x\in\mathcal F}\widehat{g-t}(x\oplus y)^2.
$$
Every auxiliary function is supported on $\mathcal F$ and has squared norm $1$, so $\sum_{y\in\mathcal F}u_j(y)^2=2^n$ and $\sum_{y\in\mathcal F}v_k(y)^2=2^n$. It follows that
\begin{equation}\label{eq:internal-lower-bound}
\sum_{x,y\in\mathcal F}\widehat{g-t}(x\oplus y)^2  \ge t^2r+(1-t)^2s=t^2|\mathcal F\setminus\mathcal G|+(1-t)^2|\mathcal F\setminus\mathcal G^*|.
\end{equation}
The argument also covers $r=0$, $s=0$, or $\mathcal F=\varnothing$ by the empty-sum convention.

Because $g$ is Boolean-valued and $\E[g^*]=1-\E[g]$, we have
$$
\|g-t\|_2^2=(1-t)^2\E[g]+t^2(1-\E[g])=(1-t)^2\E[g]+t^2\E[g^*].
$$
Combining~\eqref{eq:split-energy} and~\eqref{eq:internal-lower-bound}, and then expanding the two set differences, gives
$$
\begin{aligned}
\sum_{\substack{x\in\mathcal F,y\notin\mathcal F}}\widehat{g-t}(x\oplus y)^2
&\le |\mathcal F|\,\|g-t\|_2^2-t^2|\mathcal F\setminus\mathcal G|-(1-t)^2|\mathcal F\setminus\mathcal G^*|\\
&=t^2\bigl(|\mathcal F\cap\mathcal G|-|\mathcal F|\E[g]\bigr)+(1-t)^2\bigl(|\mathcal F\cap\mathcal G^*|-|\mathcal F|\E[g^*]\bigr)\\
&=2^n\bigl(t^2\Cov(f,g)+(1-t)^2\Cov(f,g^*)\bigr).
\end{aligned}
$$
In the last step we used $|\mathcal F|=2^n\E[f]$, $|\mathcal F\cap\mathcal G|=2^n\E[fg]$ and $|\mathcal F\cap\mathcal G^*|=2^n\E[fg^*]$. 
Combining this bound with~\eqref{eq:cross-support} and~\eqref{eq:center-cross}
proves~\eqref{eq:cross-parameter} for every real $t$.
\end{proof}

\subsection{Proofs of Theorem~\ref{thm:harmonic} and Corollary~\ref{conj:FKKK-antipodal}}

We begin with the following estimate.

\begin{lemma}\label{lem:flip}
Let $f:\cube^n\to\cube$ be increasing and let $\varnothing\ne T\subseteq[n]$. Then
$$
\max_{i\in T}\Inf_i[f]
\le\E_x\bigl[(f(x)-f(x\oplus \mathbf{1}_T))^2\bigr]=4\sum_{\substack{S\subseteq[n]\\|S\cap T|\text{ odd}}}\hat f(S)^2.
\eqno\eqtag{eq:flip}
$$
\end{lemma}

\begin{proof}
Fix $i\in T$. The change of variables $y=x\oplus \mathbf{1}_T$ preserves the uniform measure and satisfies $2x_i-1=-(2y_i-1)$. Hence 
$$
\E_x[(2x_i-1)f(x\oplus \mathbf{1}_T)]=-\E_y[(2y_i-1)f(y)]=-\E_x[(2x_i-1)f(x)].
$$
Using this identity and~\eqref{eq:influence-sign}, we obtain
$$
\begin{aligned}
\Inf_i[f]
=\E_x\bigl[(2x_i-1)(f(x)-f(x\oplus \mathbf{1}_T))\bigr]\le\E_x\bigl[|f(x)-f(x\oplus \mathbf{1}_T)|\bigr]=\E_x\bigl[(f(x)-f(x\oplus \mathbf{1}_T))^2\bigr].
\end{aligned}
$$
The last equality holds because the difference belongs to $\{-1,0,1\}$. Taking the maximum over $i\in T$ proves the inequality in~\eqref{eq:flip}. The Fourier identity is exactly~\eqref{eq:flip-fourier}.
\end{proof}

\begin{proof}[Proof of Theorem~\ref{thm:harmonic}]
Multiply~\eqref{eq:flip} by $\hat{g}(T)^2$ and sum over nonempty $T$. This gives 
\begin{equation}\label{eq:weighted-flip}
\sum_{\varnothing\ne T\subseteq[n]}\hat g(T)^2\max_{i\in T}\Inf_i[f]\le 4\sum_{T\subseteq[n]}\hat g(T)^2\sum_{\substack{S\subseteq[n]\\|S\cap T|\text{ odd}}}\hat f(S)^2. 
\end{equation}
We have included $T=\varnothing$ on the right-hand side because its contribution is zero. 

Suppose first that $\Cov(f,g)+\Cov(f,g^*)>0$. Taking $t=\Cov(f,g^*)/(\Cov(f,g)+\Cov(f,g^*))$ in Theorem~\ref{thm:cross}, we obtain \begin{equation*}
2\sum_{\substack{S,T\subseteq[n] \\ |S\cap T|\text{ odd}}}\hat g(T)^2\hat f(S)^2\le \frac{\Cov(f,g)\Cov(f,g^*)}{\Cov(f,g)+\Cov(f,g^*)},
\end{equation*}
which, together with~\eqref{eq:weighted-flip}, proves~\eqref{eq:harmonic}. If $\Cov(f,g)+\Cov(f,g^*)=0$, then $\Cov(f,g)=\Cov(f,g^*)=0$ by nonnegativity. Theorem~\ref{thm:cross}  and~\eqref{eq:weighted-flip} then force the nonnegative left-hand side of~\eqref{eq:harmonic} to vanish, as needed. 
\end{proof}

\begin{proof}[Proof of Corollary~\ref{conj:FKKK-antipodal}]
Antipodality gives $g=g^*$, $\E[g]=\frac{1}{2}$, and  $\Var(g)=\frac{1}{4}$. Theorem~\ref{thm:harmonic} reduces to $\Cov(f,g)\ge\sum_{\varnothing\ne S\subseteq[n]}\hat g(S)^2\max_{i\in S}\Inf_i[f]$, while Parseval gives
\[
\sum_{\varnothing\ne S\subseteq[n]}\hat g(S)^2\max_{i\in S}\Inf_i[f]\ge\min_{i\in[n]}\Inf_i[f]\sum_{\varnothing\ne S\subseteq[n]}\hat g(S)^2=\min_{i\in[n]}\Inf_i[f]\cdot\Var(g)
=\frac14\min_{i\in[n]}\Inf_i[f].\tag*{\qedhere}
\]
\end{proof}

\section{Chv\'atal's conjecture}\label{sec:known-equivalence}

We now recall the equivalence established by Friedgut, Kahn, Kalai and Keller~\cite[Proposition~2.1]{FKKK2018correlation}. We include the details only to make the deduction from Theorem~\ref{thm:harmonic} self-contained.

A family $\mathcal{B}\subseteq 2^{[n]}$ is \emph{maximal intersecting} if it is intersecting and is not properly contained in any other intersecting family in $2^{[n]}$. It is \emph{antipodal} if it contains
exactly one of $S$ and $[n]\setminus S$ for every $S\subseteq[n]$.

\begin{prop}\label{lem:known-maximal-antipodal}
A family $\mathcal B\subseteq2^{[n]}$ is maximal intersecting if and only if it is increasing and antipodal. In this case $|\mathcal B|=2^{n-1}$.
\end{prop}

\begin{proof}
A maximal intersecting family is increasing, since adjoining
supersets preserves intersection. For an increasing family
$\mathcal B$, a set $S$ intersects every member of $\mathcal B$
if and only if $[n]\setminus S\notin\mathcal B$:
any member disjoint from $S$ is contained in its complement.
Thus $\mathcal B$ is maximal intersecting precisely when
$S\in\mathcal B$ if and only if $[n]\setminus S\notin\mathcal B$,
which is antipodality. Counting complementary pairs gives
$|\mathcal B|=2^{n-1}$.
\end{proof}

Let $\mathcal{D}$ be hereditary, let $\mathcal{B}$ be maximal intersecting, and set $f=1-\mathbbm{1}_{\mathcal{D}}$ and $g=\mathbbm{1}_{\mathcal{B}}$. Then $f,g$ are increasing, and Proposition~\ref{lem:known-maximal-antipodal} gives that $g$ is antipodal. Direct counting gives
\begin{equation*}
\Cov(f,g)=\frac{1}{2^{n}}\left(\frac{|\mathcal{D}|}{2}-|\mathcal{D}\cap\mathcal{B}|\right).
\end{equation*}
Deleting $i$ matches each member of $\mathcal{D}(i)$ with a member of $\mathcal{D}$ omitting $i$. The number of $i$-edges leaving $\mathcal{D}$ is therefore $|\mathcal{D}|-2|\mathcal{D}(i)|$. Each edge contributes $2/2^n$ to the influence, so
\begin{equation*}
\Inf_i[f]=2^{1-n}\bigl(|\mathcal{D}|-2|\mathcal{D}(i)|\bigr).
\end{equation*}
Subtracting gives the useful identity
\begin{equation}\label{eq:equivalence}
\Cov(f,g)-\frac14\min_{i\in [n]}\Inf_i[f]
=2^{-n}\left(\max_{i\in [n]}|\mathcal{D}(i)|-|\mathcal{D}\cap\mathcal{B}|\right).
\end{equation}
\begin{proof}[Proof of Theorem~\ref{conj:chvatal}]
Extend $\mathcal{F}$ to a maximal intersecting family $\mathcal B$ in $2^{[n]}$; this is possible by finiteness, even when $\mathcal{F}$ is empty. Corollary~\ref{conj:FKKK-antipodal} and~\eqref{eq:equivalence} imply
\[
|\mathcal F|\le|\mathcal D\cap\mathcal B|\le\max_i|\mathcal D(i)|.\tag*{\qedhere}
\]
\end{proof}
For the converse implication, start with increasing Boolean $f,g$ with $g$ antipodal and take $\mathcal{D}=\supp (1-f)$ and $\mathcal{B}=\supp g$. Proposition~\ref{lem:known-maximal-antipodal} makes $\mathcal D\cap\mathcal B$ an intersecting subfamily of $\mathcal D$. The star bound makes the right-hand side of~\eqref{eq:equivalence} nonnegative, giving~\eqref{eq:antipodal-consequence}.

\section{Sharpness and further consequences}\label{sec:consequences}
\paragraph{Sharpness.}
Write $\W(f,g):=\sum_{\varnothing\ne S\subseteq[n]}\hat g(S)^2\max_{i\in S}\Inf_i[f]$. We also write $\mathrm{AND}_n(x)=\prod_{i=1}^nx_i$ and $\mathrm{OR}_n(x)=1-\prod_{i=1}^n(1-x_i)$. 
\begin{prop}\label{prop:sharpness}
If $f(x)=\prod_{i=1}^n x_i$, then equality holds in~\eqref{eq:harmonic} for every increasing Boolean $g$. The factor $1/4$ in~\eqref{eq:antipodal-consequence} is optimal.
\end{prop}
\begin{proof}
The constant cases for $g$ are immediate. Otherwise set $a=2^{-n}$ and $b=\E[g]$. Monotonicity gives $g(\mathbf{0})=0$, $g(\mathbf{1})=1$, and hence
\[
\Inf_i[f]=2a,\qquad \Cov(f,g)=a(1-b),\qquad \Cov(f,g^*)=ab.
\]
Parseval's identity gives $\W(f,g)=2ab(1-b)$, which equals the harmonic mean in~\eqref{eq:harmonic}. If $g$ is antipodal, then $b=\frac{1}{2}$, so $\Cov(f,g)=\frac{a}{2}=\frac14\min_i\Inf_i[f]$.
\end{proof}
\begin{remark}
For $f=\mathrm{AND}_2$ and $g=\mathrm{OR}_2$,
\[
\Cov(f,g)=\frac1{16},\qquad \Cov(f,g^*)=\frac3{16},\qquad
\W(f,g)=\frac3{32}>\Cov(f,g).
\]
Thus $\Cov(f,g)\ge\W(f,g)$ fails without antipodality.    
\end{remark}

\paragraph{Kleitman's weighted bound.}
The ordered spectral term yields coefficients that depend only on the maximal intersecting family, not on the hereditary family or the weight being tested. This gives Kleitman's strengthening~\cite{Kleitman1979} in its weighted formulation~\cite{Fishburn1988}; see also~\cite[Section~3]{FKKK2018correlation}. Cary~\cite{Cary2026} recently studied related formulations using cubical cohomology and convex optimization.

Fix a permutation
$(i_1,\ldots,i_n)$ of $[n]$, and let $\prec$ be the total order defined by
$i_j\prec i_k$ if and only if $j<k$. For every nonempty $S\subseteq[n]$,
let $\max_\prec S$ denote its last element in this order; explicitly,
$\max_\prec S=i_k$, where $k=\max\{j:i_j\in S\}$.
\begin{prop}\label{prop:weighted}
Let $\mathcal B\subseteq2^{[n]}$ be maximal intersecting, and set $g=\mathbbm{1}_{\mathcal B}$. Define
\[
\lambda_i
=
4\sum_{\substack{\varnothing\ne S\subseteq[n]\\
                  \max_\prec S=i}}\hat g(S)^2.
\]
Then $\lambda_i\ge0$ for every $i\in [n]$, $\sum_{i=1}^n\lambda_i=1$, and for every nonnegative decreasing function $\omega:2^{[n]}\to\mathbb R$,
\begin{equation}\label{eq:weightedstar}
\sum_{S\in\mathcal B}\omega(S)
\le\sum_{i=1}^n\lambda_i\sum_{S\ni i}\omega(S)
\le\max_{i\in[n]}\sum_{S\ni i}\omega(S).
\end{equation}
Consequently, among all intersecting families, some star has maximum total $\omega$-weight.
\end{prop}

\begin{proof}
By Proposition~\ref{lem:known-maximal-antipodal}, $g$ is increasing
and antipodal. Thus $\E[g]=\frac{1}{2}$ and $\Var(g)=\frac{1}{4}$. Each nonempty
$S\subseteq[n]$ has a unique $\prec$-maximum, so Parseval's identity gives
$$
\sum_{i=1}^n\lambda_i=4\sum_{\varnothing\ne S\subseteq[n]}\hat g(S)^2=4\Var(g)=
1.
$$
Let $q(x)=\sum_{i=1}^n\lambda_i x_i$. Then $\E[q]=\frac{1}{2}=\E[g]$. For every increasing Boolean function $f$,
$$
\Cov(f,x_i)=\frac14\bigl(\E[f\mid x_i=1]-\E[f\mid x_i=0]\bigr)=
\frac14\Inf_i[f],
$$
where the last equality uses the monotonicity and Boolean valued of $f$.
The antipodal case of Theorem~\ref{thm:harmonic} therefore gives
$$
\Cov(f,g)\ge\sum_{\varnothing\ne S\subseteq[n]}\hat g(S)^2\max_{i\in S}\Inf_i[f]\ge\sum_{\varnothing\ne S\subseteq[n]}\hat{g}(S)^2\Inf_{\max_\prec S}[f]=\frac14\sum_{i=1}^n\lambda_i\Inf_i[f]=\Cov(f,q).
$$

Now let $\mathcal{D}\subseteq2^{[n]}$ be decreasing, and set $h=\mathbbm{1}_{\mathcal{D}}$. Since $1-h$ is increasing and $\E[g-q]=0$, the preceding comparison implies
$$
0\le \Cov(1-h,g-q)=-\E[h(g-q)]=\frac{1}{2^n}\left(\sum_{i=1}^n\lambda_i|\mathcal D(i)|-|\mathcal B\cap\mathcal D|
\right).
$$
Hence $|\mathcal B\cap\mathcal D|\le
\sum_{i=1}^n\lambda_i|\mathcal D(i)|$ holds for every decreasing family $\mathcal D$.

For $t\ge0$, define $\mathcal{D}_t=\{S\subseteq[n]:\omega(S)>t\}$. Each $\mathcal D_t$ is decreasing, and nonnegativity gives
$$
\omega(S)=\int_0^\infty\mathbbm{1}_{\mathcal D_t}(S)\,dt.
$$
Hence
$$
\sum_{S\in\mathcal B}\omega(S)=\int_0^\infty|\mathcal B\cap\mathcal D_t|\,dt\le\sum_{i=1}^n\lambda_i\int_0^\infty|\mathcal D_t(i)|\,dt=
\sum_{i=1}^n\lambda_i\sum_{\substack{S\subseteq[n]\\i\in S}}\omega(S).
$$
This proves the first inequality in~\eqref{eq:weightedstar}.
The second follows because the middle expression is a convex
combination of the star weights.

Finally, every intersecting family $\mathcal{F}$ is contained in a maximal intersecting family $\mathcal B'$. Since $\omega\ge0$,
applying~\eqref{eq:weightedstar} to $\mathcal B'$ gives
$$
\sum_{S\in\mathcal{F}}\omega(S)\le
\sum_{S\in\mathcal B'}\omega(S)\le
\max_{i\in[n]}\sum_{i\in S}\omega(S).
$$
For every $i\in [n]$, the star $\{S\subseteq [n]:i\in S\}$ is itself intersecting, so a star attains the
maximum total $\omega$-weight.
\end{proof}

To compare this with~\cite[Section~3]{FKKK2018correlation}, let $h=2g-1$.
Then $h$ is an increasing $\{-1,1\}$-valued function satisfying $h(1-x)=-h(x)$, and $h_+^2=g$, where $h_+=\max\{h,0\}$. For every nonempty $S$, we have $\hat{h}(S)=2\hat{g}(S)$, and hence
$$
\lambda_i=\sum_{\substack{\varnothing\ne S\subseteq[n]\\\max_\prec S=i}}\hat{h}(S)^2.
$$
Thus Proposition~\ref{prop:weighted}  proves the $\{-1,1\}$-valued case
of~\cite[Conjecture~3.4]{FKKK2018correlation} and its equivalent formulation~\cite[Conjecture~3.5]{FKKK2018correlation}. The same argument proves the $\{-1,1\}$-valued case of~\cite[Conjecture~3.6]{FKKK2018correlation}.

\section*{Acknowledgements}
The authors thank Mengyu Cao, Ting-wei Chao, Xingtong Guo, Guowei Sun, and Haixiang Zhang for valuable discussions during the fourth ECOPRO Student Research Program, held at the Institute for Basic Science (IBS) in the summer of 2026. 

The authors acknowledge the use of AI tools. We used ChatGPT to test candidate inequalities for special classes of functions involving squared covariances of $f$ and (the dual of) $g$. In particular, ChatGPT proved the special case of the inequality in Theorem~\ref{thm:harmonic} when $f$ is a symmetric threshold function, inspiring the authors to pursue and ultimately prove the stronger inequality in Theorem~\ref{thm:cross}. All mathematical arguments and proofs in the manuscript were written and checked by the authors.

\bibliographystyle{abbrv}
\bibliography{reference1}

\end{document}